\documentclass{svjour3}
\usepackage{amsmath,amsfonts}
\usepackage{enumitem}
\usepackage{color}
\usepackage{dsfont}
\usepackage{dirtytalk}
\usepackage{graphicx}
\graphicspath{{./plots}}

\def\CC{\mathbb{C}}
\def\ci{{\rm c}}
\def\fii{\varphi}

\def\R{\mathbb{ R}}
\def\N{\mathbb{ N}}

\def\S{\mathcal{ S}}

\def\sig{\sigma}
\def\dta{\delta} 
\def\Lda{\Lambda}

\def\i{\mathrm{i}}
\def\e{\mathrm{e}}
\def\d{\mathrm{d}}

\begin{document}

\title{On  $L^p-L^{p'}$ estimates for a class of strongly damped wave equations }

\date{\today}

\author{Haidar Mohamad}

\institute{Mathematical Institute for Machine
Learning and Data Science,
KU Eichst\"att--Ingolstadt,
85049 Ingolstadt, Germany}

\maketitle
\begin{abstract}
The purpose of this paper is to obtain  a fundamental $L^p-L^{p'}$ estimate for a class of a strongly damped wave equations 
where the damping operator is given by $-\dta  \Delta$ with $\dta \geq 0$ and the constant in the estimate is independent of the damping parameter $\dta.$   The method used here is based on an estimate of  the semigroup associated to the linear equation on the Besov spaces and their elementary properties.   
\end{abstract}


\section{Introduction}\label{Intro}
Our setting is the  \emph{damped} semilinear wave equation,
written as
\begin{equation}\label{MainEq}
 \partial^2_t \psi -  (\dta_1\Delta - \dta_2) \, \partial_t \psi
  - \Delta \psi = f(\psi) \,, 
\end{equation}
where $\psi: [0,T]\times \R^n\to \CC,$ $n\in \N,$ $f$ is a nonlinear map and $\dta_1, \dta_2$ are two non-negative real numbers.
The linear part of the system is dissipative in the sense that the linear semigroup action loses
energy.  In general, damping can be \emph{strong} when the semigroup generated
by the linear part of the equation is compact (ex.   $\dta_1>0$), and \emph{weak} when
the semigroup is merely continuous.  While the time-asymptotic
behavior of weakly damped-driven equations can be almost as regular as
that of strongly damped equations \cite{Goubet:1996:RegularityAW,Goubet:2018:AnalyticityGA,OliverT:1998:AnalyticityAN}, 
the transient behaviour can be very different.  We expect that most of our proposed
results will depend on strong damping, the question of what remains in
the weakly damped case will always run in the background.

The most fundamental question that underlies any work on qualitative
properties or numerical schemes is the provision of a functional
setting in which the Cauchy problem associated with \eqref{MainEq} is
well posed.  The Cauchy problem for \eqref{MainEq} was widely studied by many
authors \cite{LinTanaka:1998:NonLinAbsdWE,Neubrander:1986:WellPosedHiOrdCauchy,Reed:1975:AbctNWE,AvilesSand:1985:NonLin2ndODE}.  Contrary to the usual method of reducing \eqref{MainEq} to a first
order system in some ``energy space'' (as, e.g., in
\cite{Reed:1975:AbctNWE,HolmMrsdn:1978:BifDiv}), Aviles and Sandefur
\cite{Sandefur:1983:ExUn2ndOrdODE,AvilesSand:1985:NonLin2ndODE} used
a factorization that allows the equation to be written in a ``second
order mild formulation'', namely an integral equation containing a
double integral involving the nonlinearity.  They assumed that $f$ is a locally Lipschitz from $H^2(\R^n)$ to
$L^2(\R^n)$ with $ n \leq 3$.  In particular, they extended global
well-posedness for the much studied nonlinearity
\begin{equation}\label{NonLin}
f(\psi) = a \, \psi + b \, |\psi|^{r-1} \, \psi
\end{equation}
 for any $r\geq 1$.
Contrary to \cite{AvilesSand:1985:NonLin2ndODE}, it has been shown in
\cite{Reed:1975:AbctNWE,Strauss:1968:DecAsym,Webb:1980:ExAsymSDampWv} that if the requirements on the nonlinearity
are reduced to local Lipschitz continuity from $H^1(\R^n)$ to
$L^2(\R^n)$ with $ n \leq 3$ and for $1\leq r \leq 3,$ the Cauchy
problem is globally well posed.  More recent results concern the
initial-boundary value problem for \eqref{MainEq} on a bounded spatial
domain \cite{HaNakagiri:2004:IdentKG}, for periodic boundary
conditions \cite{GaoGuo:2004:TimePer2DKG}, and on non-cylindrical
domains \cite{HaPark:2011:GblExUnifDecKG}.  
The Cauchy problem for the  un-damped equation ($\dta_1=\dta_2=0$) 
with the  class of nonlinearity given by \eqref{NonLin} 
has gained more  attention in the literature  \cite{Lions:1969:MethResNL,Strauss:1970:WkSol,GlasseyTsutsumi:1982:UnqWkSol,HeinzWahl:1975:FEBrowderNLW,Wahl:1975:NLWellip,Brenner:1979:GblSthSol,BrennerWahl:1981:GblExWE}. 
The classical challenge in this context is to establish well-posedness results in the energy
space $(H^1(\R^n) \cap L^{r+1}(\R^n) ) \times L^2(\R^n)$ for wide range of the spacial dimension $n$
and the nonlinearity  \eqref{NonLin} (namely,  for wide range of $r.$).   
Our two main long-term goals after this paper are:
\begin{enumerate}
\item to establish global well posedness results in the energy space for the Cauchy problem related 
to the strongly damped Klein-Gordon equation 
\begin{equation}\label{StrongDdKG}
 \partial^2_t \psi -  \dta\Delta  \, \partial_t \psi
  - \Delta \psi  = f(\psi) \,, 
\end{equation}
for arbitrary  $n$ and for relatively large choice of $r$ in \eqref{NonLin}. 
\item  To investigate the zero-damping limit ($\dta \to 0$) in the energy space.  
\end{enumerate}
Therefore, the main result of this paper is 
a preliminary step toward this goal in which we prove  $L^p-L^q$ estimate related to the semi-group generated by the linear part of  \eqref{StrongDdKG} with some uniformity  with respect to the parameter $\dta \geq 0.$ Such estimates have gained more attention for the weakly damped equation 
\eqref{MainEq} corresponding to $\dta_1= 0$ and $ \dta_2=1.$ Namely, 
\begin{equation}\label{WeakgDdKG}
 \partial^2_t \psi +  \, \partial_t \psi
  - \Delta \psi =f(\psi) \,, 
\end{equation}
 In the literature,  they are called the $L^p-L^q$  estimates only when the action of the semi-group is considered on the
 spacial $L^p$ space. That is mostly the case of linear equations \cite{Matsumura:1976:AsympSLWE,Nishihara:2003:LpLqEstimWDWE3d,TakafumiTakayoshi:2004:LpLqEstimWDWE2d,Takashi:2005:LpLqEstimWDWEOddID}.  When the estimate takes into account the time contribution in the semi-group,  it is mostly called  Strichartz type estimate which is an 
 important tool to prove local well posedness results for the nonlinear equations.  In a recent paper,  Watanabe \cite{Watanabe:2017:StrichartzEstimWDWEd23} established a Strichartz type estimate for the  nonlinear  \eqref{WeakgDdKG} with nonlinearity given by \eqref{NonLin} when $n= 2,3.$ Using a low-high frequency splitting for  the semi-group generated by the linear part of \eqref{WeakgDdKG},  Inui extended the results of Watanabe to higher dimensions $n\geq 3.$  Our $L^p-L^q$ estimate related to the 
 strongly damped wave equation is achieved via an estimate  in the so called Besov spaces (Theorem \ref{ThmEstimBesov}) and it
 exhibits  an exponential decay in times.  The remainder of the paper is structured as follows.  In Section 2,  we state the definition of the 
 Besov spaces and we state some of their elementary properties.  In Section 3, we state our main results.  Namely, Theorem \ref{ThmEstimBesov} and Corollary \ref{CorLpLqEstim}.  We conclude by Section 4 with some discussion.

\section{Preliminaries}\label{Pre}
In this section we collect some basic facts about the Besov spaces. The reader is referred to \cite{Triebel:1992:TheoFnSpcII,BahouriCheminDanchin:2011:FourierNPDEs,Loefstroem:1970:BesovSpcTheoApprox} for a general framework and a sketch of the main results.  The Besov spaces can be defined using the dyadic decomposition.  For this purpose,  let $\varphi \in C^\infty_\ci(\R^n)$ be such
$ {\rm supp}(\varphi) \subseteq \left\{x; \,  \frac 12<  |x| < 2\right\}$ and denote 
$$\varphi_j(x) = \varphi\left(2^{-j}x\right)\quad\text{for} \quad j\in \N^\ast\quad \text{and}\quad \varphi_0 = 1-\sum_{j=1}^\infty \varphi_j.$$
For any $\sig\in \R$ and $0< p, q \leq \infty,$ define the  Besov space 

$$B_{pq}^\sig:=  \left\{ v \in \S'(\R^n);\,\,  \|v\|_{B_{pq}^\sig} := \left(\sum_{j=0}^\infty 2^{js q} \left\| \mathfrak{F}^{-1} (\varphi_j \hat v)\right\|_p^q  \right)^{\frac 1q} <\infty \right\},$$
where $\S'(\R^n)$ denotes the space of tempered distributions and $\mathfrak{F}^{-1}$ is the  inverse Fourier transform.
For $p>0,$ we denote $p' >0$ the conjugate of $p$ which satisfies 
$$\frac 1p + \frac{1}{p'} = 1$$
The following elementary two lemmas will be needed in the proof of the main result
\begin{lemma}\label{LemLp'pEstim}
Fix $h\in L^\infty(\R^n),$ $1\leq p\leq 2$ and $\sig \in \R$  and assume that there exists $C>0$ such that 
\begin{equation}\label{Estim1}
\left\| \mathfrak{F}^{-1}(h\,\varphi_j \hat v)\right\|_{p'} \leq C \|v\|_p \quad \text{for any} \quad v\in L^p(\R^n) \quad \text{and} \quad 
j\in \N.
\end{equation}
Then,  for any $q\geq 1$ there exists a constant $\tilde C>0$ independent of $h$ such that 
$$\left\| \mathfrak{F}^{-1}(h\,\hat v)\right\|_{B_{p' q}^\sig} \leq \tilde C C \| v\|_{B_{p q}^\sig}\quad \forall v\in B_{p q}^\sig.$$
\end{lemma}
 \begin{proof}
 Set $\varphi_{-1} =0.$ Since $\varphi_j \varphi_k =0$ for $|j-k| >1$ and $\sum_{k=0}^\infty \varphi_k=1$ we have 
 $$\varphi_j = \varphi_j \sum_{j=0}^\infty \varphi_k = \fii_j (\fii_{j-1} + \fii_j + \fii_{j+1}),\quad j=0,1, \cdots.$$
 
 Let now $v\in B_{p q}^\sig$ and denote $$v_j := \mathfrak{F}^{-1}( (\fii_{j-1} + \fii_j + \fii_{j+1})\hat v)\in L^p(\R^n).$$
 Applying \eqref{Estim1} on $v_j,$ we find 
 \begin{align*}
  \left\| \mathfrak{F}^{-1}(h\,\varphi_j \hat v)\right\|^q_{p'} &= \left\| \mathfrak{F}^{-1}(h\,\varphi_j \hat v_j)\right\|^q_{p'}\notag\\
  & \leq C^q \|v_j\|^q_p \notag \\
  & \leq C^q 3^{q-1} \left( \left\| \mathfrak{F}^{-1}(h\,\varphi_{j-1} \hat v)\right\|^q_p+  \left\| \mathfrak{F}^{-1}(h\,\varphi_j \hat v)\right\|^q_p+ \left\| \mathfrak{F}^{-1}(h\,\varphi_{j+1} \hat v)\right\|^q_p\right).\notag\\
\end{align*}
Thus,  multiplying by $2^{j \sig q}$ and summing over $j,$ we find that
$$\left\| \mathfrak{F}^{-1}(h\,\hat v)\right\|_{B_{p' q}^\sig} \leq C 3^{q-1}(2^{\sig q}+2^{-\sig q}+1)\| v\|_{B_{p q}^\sig}$$
which is the desired estimate. 
 \end{proof}

\begin{lemma}\label{LemLp'pInterp}
Fix $h\in L^1(\R^n)\cap L^\infty(\R^n)$ and let $C_1 , C_\infty >0$ such that 
$$\left\| \mathfrak{F}^{-1}h\right\|_\infty \leq C_1\quad \text{and} \quad \|h\|_\infty \leq C_\infty.$$
Then,   we have 
\begin{equation}
\left\| \mathfrak{F}^{-1}(h\,\hat v)\right\|_{p'} \leq  C_1^{1-\alpha} C_\infty^\alpha \|v\|_p,\quad \forall v\in L^1(\R^n)\cap L^2(\R^n),
\end{equation}
where  $p\in [1,2]$  and $\alpha = \frac 2p -1$
\end{lemma}
\begin{proof}
Using Young's convolution inequality and Parseval's identity, we find that 
\begin{equation}\label{T1toInfty}
\left\| \mathfrak{F}^{-1}(h\,\hat v)\right\|_\infty \leq C_1 \|v\|_1
\end{equation}
and 
\begin{equation}\label{T2to2}
\left\| \mathfrak{F}^{-1}(h\,\hat v)\right\|_2 \leq C_\infty  \|v\|_2
\end{equation}
Defining $T(v) = \mathfrak{F}^{-1}(h\,\hat v),$ we find from \eqref{T1toInfty} and \eqref{T2to2} that 
$$\|T\|_{L^1 \to L^\infty} \leq C_1\quad \text{and} \quad \|T\|_{L^2 \to L^2} \leq C_\infty$$
Using the Riesz-Thorin  interpolation theorem \cite{Thorin:1948:ConvThmGen}, we find that 
for any $\alpha \in [0,1],$  $T$  boundedly maps $L^p(\R^n)$
into $L^{p'}(\R^n),$ with $$\frac{1}{p} =  \alpha  + \frac{1-\alpha}{2} = \frac{1+ \alpha}{2} $$
and we have $$\|T\|_{L^p \to L^{p'}} \leq \|T\|^{1-\alpha}_{L^1 \to L^\infty} \, \|T\|^\alpha_{L^2 \to L^2},$$
which completes the proof. 
\end{proof}
Let now $v\in C_\ci^\infty(\R^n)$ and $\Lambda \in C^\infty(\R^n)$ be real valued on a neighborhood  of ${\rm supp}(v).$ 
We are interested in proving a Van der Corput type estimate of  
\begin{equation}\label{LittmanInteg}
t\mapsto \mathfrak{F}^{-1}\left(\e^{\i t \Lambda} v\right)(\xi)= \int_{\R^n} \e^{2 \pi \i \xi \cdot x} \, \e^{\i t \Lambda(x)} v(x)\d x.
\end{equation}
To this end, we will use the result of Littman \cite{Littman:1963:FourierSurMes} related to the behavior at 
$\infty$ of the Fourier transform of a surface-carried measure.  First, we write \eqref{LittmanInteg} as an integral over 
a smooth $n-$surface in $\R^{n+1}.$ Define the surface 
$$S= \left\{X:= \left(x_1,\cdots, x_n, \frac{\Lambda(x)}{2\pi}\right); \, x:= (x_1,\cdots, x_n)\in {\rm supp}(v)\right\}.$$ 
Then, the surface measure on $S$ is given by $$\d S(X) = \sqrt{1+|\nabla \Lambda(x)|^2} \d x,$$ 
and we can write 
$$ \mathfrak{F}^{-1}\left(\e^{\i t \Lambda} v\right)(\xi)= \int_S \e^{2 \pi \i Y\cdot X} \, V(X)\d S(X),$$
where $V(X) = \frac{v(x)}{\sqrt{1+|\nabla \Lambda(x)|^2}}$ and $Y= (\xi, t).$ Littman \cite{Littman:1963:FourierSurMes} assumed that 
at each point of $S,$  there are atleast $k$  of the $n$ principal curvatures are different from zero. This assumption is reflected on 
$\Lambda$ as follows 
\begin{theorem}[\cite{Littman:1963:FourierSurMes}]\label{ThmLittman}
Keep the above settings for $\Lambda$ and $v$ and assume that the rank of ${\rm Hess}(\Lambda)$ is at least $k$ on ${\rm supp}(v).$
Then, there exists a constant $C$ depending on  bounds of the derivatives of $\Lambda,$ on a  lower bound of the maximum  of the absolute values of the  $k-$ order minors of ${\rm Hess}(\Lambda)$  on ${\rm supp}(v)$ and on $v$ such that 
for some $m \in \N$
\begin{equation}
\left\|\mathfrak{F}^{-1}\left(\e^{\i t \Lambda} v\right)\right\|_\infty \leq C(1+|t|)^{-\frac k2}\sum_{|l| \leq m}\left\|{\rm D}^l v\right\|_1.
\end{equation}
\end{theorem}

For $x\in \R^n,$ denote $r= |x|$ and let $\lambda = \lambda(r)$ be a radial function on an open set $B\subset \R^n.$  
A direct computation shows that 
$$\frac{\partial^2 \lambda }{\partial x_j \partial x_k} = \frac{\lambda '}{r} \delta_{jk} +\left(\frac{\lambda ''}{r^2}- \frac{\lambda '}{r^3}\right) x_j x_k,$$
which implies that 
$${\rm Hess}(\lambda ) = \frac{\lambda '}{r} \left( I + \frac{\mu}{r^2}x x^\top\right),\quad \text{with} \quad \mu = \frac{r}{\lambda '}\left(\lambda '' - \frac{\lambda '}{r}\right).$$
Thus, we get $${\rm Det}({\rm Hess}(\lambda ))= \left( \frac{\lambda '}{r}\right)^n (1+\mu) =  \left( \frac{\lambda '}{r}\right)^{n-1} \lambda ''.$$
Let now $\delta \in [0, \frac{1}{2\sqrt{2}}]$ and consider the radial function $\lambda _\delta(r)= r\sqrt{1-\delta^2 r^2}$ defined on the 
set $A:=\left\{x; \,  \frac 12<  |x| < 2\right\}.$  For $r\neq 0,$ we have 
$$\lambda '(r) = \frac{1-2\delta^2 r^2}{\sqrt{1-\delta^2 r^2}}, \,\, \lambda ''(r)=\frac{ r \delta^2 (2 r^2 \delta^2-3)}{(1 - r^2 \delta^2)^{\frac 32}}.$$
It's easy then to check that for $\delta \in (0, \frac{1}{2\sqrt{2}}],$ we have 
${\rm Det}({\rm Hess}(\lambda _\delta))\neq 0$ and ${\rm rank}({\rm Hess}(\lambda _\delta)) = n.$ In the case where $\delta =0,$ we have 
$${\rm Hess}(f_0) = \frac 1r \left( I - \frac{1}{r^2}x x^\top\right).$$
A direct computation shows that $r\,{\rm Hess}(\lambda _0)$ is an idempotent matrix.  Since the rank of an idempotent matrix equals its trace, we can write $${\rm rank}({\rm Hess}(\lambda _0(|x|))) =  {\rm rank}(|x|\,{\rm Hess}(\lambda _0(|x|))) = {\rm Tr}\left( I - \frac{1}{r^2}x x^\top\right) = n-1,$$
for any $x\in A.$ On the other hand, all derivatives $\frac{\partial \lambda }{\partial x_k}$ are bounded on $A$  uniformly with $\delta \in [0, \frac{1}{2\sqrt{2}}],$  and the maximum of the absolute values of the $n-1$order minors of ${\rm Hess}(\lambda )$ has a positive lower bound on $A$ independent of $\delta.$
Using now Theorem \ref{ThmLittman}, we get 
\begin{corollary}\label{CorLittman}
For any $v\in C_\ci^\infty(\R^n)$ with ${\rm supp}(v)\subseteq A,$ there exists a constant $C>0$ such that
for some $m\in \N$ 
\begin{equation}
\left\|\mathfrak{F}^{-1}\left(\sin(t \lambda _\delta) v\right)\right\|_\infty \leq C(1+|t|)^{-\frac{n-1}{2}}\sum_{|l| \leq m}\left\|{\rm D}^l v\right\|_1,
\end{equation}
uniformly with $\delta \in [0, \frac{1}{2\sqrt{2}}].$
\end{corollary}

\section{Main results}
Writing \eqref{StrongDdKG} as a first
order system, we find that its  abstract mild formulation is given by 
\begin{equation}\label{MildForm}
\Psi(t+\tau) =  \mathcal{T_\dta}(t) \Psi_0 + \int_0^\tau  \mathcal{T_\dta}(\tau - s) F(\Psi(t + s))\d s,
\end{equation}
where $\Psi := \left(\begin{matrix}\psi\\ \partial_t \psi \end{matrix}\right),$ 
$F(\Psi)  := \left(\begin{matrix}0\\  f(\psi) \end{matrix}\right)$ and 
 $ \mathcal{T_\dta}(t)$ is the semi-group generated by the linear part of the equation and is given by 
\begin{gather*}
  \mathcal{T_\dta}(t) := \e^{\dta t \Delta  }\left( \begin{matrix}
  \cosh(t \Lda_\dta) - \dta \Delta \Lda_\dta^{-1} \sinh(t \Lda_\dta) &
  \Lda_\dta^{-1} \sinh(t \Lda_\dta) \\  \Delta \Lda_\dta^{-1} \sinh(t
  \Lda_\dta) &   \cosh(t \Lda_\dta) + \dta \Delta \Lda_\dta^{-1}
  \sinh(t \Lda_\dta)\end{matrix}\right) , \\
  \Lda_\dta := \sqrt{\dta^2 \Delta^2 + \Delta}.
\end{gather*}
Our main result is the following 
\begin{theorem}\label{ThmEstimBesov}
For any $(\sig, \alpha, q) \in \R^3,$  satisfying  $\frac{2n+2}{n+3}\leq p\leq 2,$ $q\geq 1$ and $\sig>0,$ there exists  a  constant 
$ C>0$ independent of $\dta$  such that 

\begin{equation}\label{MainEstim}
\left\| \e^{\dta t \Delta}  \Lda_\dta^{-1} \sinh(t \Lda_\dta) v \right\|_{B_{p' q}^\sig} \leq  C t^{1-2n(\frac 1p -\frac 12 )} \max\left( t^{\sig}, t^{-\sig}\right)  \| v\|_{B_{p q}^\sig}\quad \forall v\in B_{p q}^\sig.
\end{equation}
\end{theorem}
\begin{proof}
First,  we prove \eqref{MainEstim} for $t=1.$
Denote $\beta_\delta(r) = r\sqrt{\delta^2 r^2-1}$ and  $\delta_j = \delta 2^j.$   Then we have 
$$ \left\| \mathfrak{F}^{-1}\left( \e^{-\dta  \lambda^2_0}\frac{\sinh(  \beta_\delta )}{\beta_\delta} \varphi_j\right) \right\|_\infty = 2^{jn}\left\| \mathfrak{F}^{-1}\left( \e^{-\dta 2^{2j} \lambda^2_0}\frac{\sinh( 2^j \beta_{\delta_j} )}{2^j\beta_{\delta_j}} \varphi\right) \right\|_\infty.$$
We consider first all $j\in \N^\ast$ such that  $\dta_j \leq \frac{1}{2 \sqrt{2}}$.  We have 
$$\frac{\sinh( 2^j  \beta_{\delta_j} )}{\beta_{\delta_j}} =\frac{\sin( 2^j\lambda_{\delta_j} )}{\lambda_{\delta_j}}\quad \text{on} \quad A.$$
The function $\frac{\e^{-\dta 2^{2j} \lambda^2_0} \varphi}{\lambda_{\delta_j}}$ is in $C_\ci^\infty(\R^n)$ with support in $ A.$
Moreover,  for any $m\in \N,$ there exists  $C_1 >0$   depending only on $\varphi$ such that 
$$\sum_{|l|\leq m}\left\| {\rm D}^l \left(\frac{\e^{-\dta 2^{2j} \lambda^2_0} \varphi}{\lambda_{\delta_j}} \right)\right\|_1 \leq   C_1.$$
Hence,  Corollary~\ref{CorLittman} implies 
\begin{equation}\label{PreEstim1}
\left\| \mathfrak{F}^{-1}\left( \e^{-\dta 2^{2j} \lambda^2_0}\frac{\sinh( 2^j \beta_{\delta_j} )}{2^j\beta_{\delta_j}} \varphi\right) \right\|_\infty \leq C_1 2^{-\frac 12 (n+1)j}
\end{equation}
Let now $j\in \N^\ast$ such that $\dta_j > \frac{1}{2 \sqrt{2}}.$  We can write 
$$\left\| \mathfrak{F}^{-1}\left( \e^{-\dta 2^{2j} \lambda^2_0}\frac{\sinh( 2^j \beta_{\delta_j} )}{2^j\beta_{\delta_j}} \varphi\right) \right\|_\infty \leq \left\| \e^{-\dta 2^{2j} \lambda^2_0}\frac{\sinh( 2^j \beta_{\delta_j} )}{2^j\beta_{\delta_j}} \right\|_{L^\infty(A)}\|\varphi\|_1.$$
There exist two constants $C_2, d >$ independent of $\dta$ such that 
$$\left\| \e^{-\dta 2^{2j} \lambda^2_0}\frac{\sinh( 2^j \beta_{\delta_j} )}{2^j\beta_{\delta_j}} \right\|_{L^\infty(A)} \leq C_2 2^{-2 j}.$$
Thus, we get finally 
\begin{equation}\label{phi_j1}
\left\| \mathfrak{F}^{-1}\left( \e^{-\dta \lambda^2_0}\frac{\sinh(  \beta_\delta )}{\beta_\delta} \varphi_j\right) \right\|_\infty  \leq  \begin{cases}
C_1 2^{\frac 12 (n-1)j}&\text{for  $\dta_j \leq \frac{1}{2 \sqrt{2}}$}\\
C_2 2^{(n-2) j} &\text{for  $\dta_j > \frac{1}{2 \sqrt{2}}$}.
\end{cases}
\end{equation}
On the other hand,  there exist $C_3, C_4 >$ depending on $\varphi$ such that 
\begin{equation}\label{phi_j2}
\left\|  \e^{-\dta \lambda^2_0}\frac{\sinh( \beta_\delta )}{\beta_\delta} \varphi_j \right\|_\infty \leq  \begin{cases}
C_3 2^{-j}&\text{for  $\dta_j \leq \frac{1}{2 \sqrt{2}}$}\\
C_4 2^{-2 j} &\text{for  $\dta_j > \frac{1}{2 \sqrt{2}}$}.
\end{cases}
\end{equation}
Finally, it is easy to check that \eqref{phi_j1} and \eqref{phi_j2} are also valid for $j=0$ with upper bounds independent of $\dta. $
Then, using Lemma \ref{LemLp'pInterp}, we find that there exists  $C> 0$ independent of $\dta$ such that 
\begin{equation}\label{Estim2}
\left\| \mathfrak{F}^{-1}\left( \e^{-\dta 2^{2j}\lambda^2_0}\frac{\sinh(\beta_{\delta} )}{\beta_{\delta}} \varphi_j \hat v\right) \right\|_{p'}
\leq C  \|v\|_p,
\end{equation}
provided that 
$$-(1-\alpha)+ \frac  \alpha2 (n-1) \leq 0,\quad -2(1-\alpha)+ \alpha  (n-2) \leq 0$$
and $$p :=  \frac{2}{\alpha+1} $$ 
which equivalently means $\frac{2n+2}{n+3}\leq p\leq 2.$  Finally,  estimate \eqref{MainEstim} for $t=1$ follows by the application of Lemma \ref{LemLp'pEstim}.  Namely, we get 
\begin{equation}\label{MainEstimt=1}
\left\| \e^{\dta \Delta}  \Lda_\dta^{-1} \sinh( \Lda_\dta) v \right\|_{B_{p' q}^\sig} \leq  \tilde C C  \| v\|_{B_{p q}^\sig}\quad \forall v\in B_{p q}^\sig,
\end{equation}
where $\tilde C$ and $C$ don't depend on $\dta.$
Let's now $t>0.$ Noting that 
$$t \beta_\dta(r) = \beta_{\dta/t}(t\,r)$$
and using the elementary property 
$$\widehat{f(t\cdot)}(\xi)  = t^{-n} \hat f\left(t^{-1}\xi \right),$$
we apply the change of variable $t\,\xi \to \xi$,  so that 
 \begin{align*}
  \mathfrak{F}^{-1}\left( \e^{-\dta t\,  \lambda^2_0}\frac{\sinh( t\, \beta_\delta )}{\beta_\delta} \hat v\right)(x) &=
t\, \int_{\R^n} \e^{2\pi \i x\cdot \xi -\frac{\dta}{t} |t \xi|^2}\frac{\sinh( \beta_\frac{\dta}{t}( |t \xi|) )}{\beta_\frac{\dta}{t}(|t\xi|)} \hat v(\xi) \d \xi\\
&= t^{1-n}  \int_{\R^n} \e^{2\pi \i t^{-1}\, x\cdot \xi -\frac{\dta}{t} |\xi|^2}\frac{\sinh( \beta_\frac{\dta}{t}( |\xi|) )}{\beta_\frac{\dta}{t}(|\xi|)} \hat v(t^{-1}\xi) \d \xi\\
& =  t^{1-2 n } \left[ \e^{\frac{\dta}{t} \Delta}  \Lda^{-1}_\frac{\dta}{t} \sinh\left( \Lda_\frac{\dta}{t}\right) v(t\cdot)\right]\left(t^{-1} x\right).
 \end{align*}
 Using now \eqref{MainEstimt=1} together with the dilation property of the Besov norm in \cite{Triebel:1992:TheoFnSpcII} (Pages 98-100), we find  that for any $\sig >0,$ there exists a constant independent of  $\dta$  such that 
 \begin{align*}
  \left\| \e^{\dta t\, \Delta}  \Lda_\dta^{-1} \sinh( t\, \Lda_\dta) v \right\|_{B_{p' q}^\sig}
   &=  t^{1-2 n } \left\|  \left[ \e^{\frac{\dta}{t} \Delta}  \Lda^{-1}_\frac{\dta}{t} \sinh\left( \Lda_\frac{\dta}{t}\right) v(t\cdot)\right]\left(t^{-1}\, \cdot\right)\right\|_{B_{p' q}^\sig}\\
   &\leq C t^{1-2 n }  \max\left( t^{\frac{n}{p'}}, t^{\frac{n}{p'}-\sig}\right)\| v(t\,\cdot)\|_{B_{p q}^\sig}\\
   &\leq C t^{1-2 n }  \max\left( t^{\frac{n}{p'}}, t^{\frac{n}{p'}-\sig}\right) \max\left( t^{-\frac np}, t^{\sig-\frac np}\right)\|v\|_{B_{p q}^\sig}\\
   & = C t^{1-2 n(\frac 1p -\frac 12) }  \max\left( t^{\sig}, t^{-\sig}\right)\|v\|_{B_{p q}^\sig},
 \end{align*}
which completes the proof
\end{proof}

\begin{remark}
The assumption $\sig>0$ is needed for the dilatation property of the Besov space.  Otherwise,  for $t= 1,$ estimate \eqref{MainEstim}
can be extended to any $\sig\in \R.$ 
\end{remark}
Using the following continuous embeddings (see Theorem 2.40 in \cite{BahouriCheminDanchin:2011:FourierNPDEs} and Theorem 15 in \cite{Taibleson:1964:TheoLipDistr})
\begin{equation}
L^p \subset B_{p 2}^0 \quad \text{and} \quad   B_{p' 2}^0 \subset L^{p'}\quad \text{for all} \quad p\in (1, 2],
\end{equation}
we find 
\begin{corollary}\label{CorLpLqEstim}
Let $n>1,$ $\frac{2n+2}{n+3}\leq p\leq 2$ and $\dta \geq 0.$ Then,  there exists 
$ C>0$ independent of $\dta$  such that 

\begin{equation}\label{StrichMainEstim}
\left\| \e^{\dta t \Delta}  \Lda_\dta^{-1} \sinh(t \Lda_\dta) v \right\|_{p'} \leq  C t^{1-2 n(\frac 1p -\frac 12) } \| v\|_p\quad \forall v\in L^p.
\end{equation}
\end{corollary}

\section{Discussion}
For $\dta >0$ the estimate in \eqref{PreEstim1} could be improved,  so that the condition on $p$ be better.  However,  in this case we lose  the uniformity with respect to $\dta$ and our technique of change of variable won't be valid. 
For bounded $t\in (0, T]$ the time depending factor in \eqref{MainEstim} could be replaced by 
$$C t^{1-2n(\frac 1p -\frac 12 )-\sig}$$
with a suitable choice of $C.$

In the estimate \eqref{StrichMainEstim} the spacial integrability is increased from $p$ to $p'.$ However,   there is no gain in the Besov smoothness $\sig$ in the original estimate \eqref{MainEstimt=1}.  This can be improved by changing  the the constant in the assumption \eqref{Estim1} of Lemma \ref{LemLp'pEstim} so that it might be proportional to $2^{j\gamma}$ with some $\gamma \leq 0.$ In this case 
it could be proved that $v\mapsto \mathfrak{F}^{-1}(h \hat v)$ maps continuously $B_{pq}^{\sig+\gamma}$ into $B_{p'q}^{\sig}$
which could be transmitted further to the estimate resulting from  \eqref{Estim2} for which 
$$\gamma =  \frac  \alpha2 (n+1) -1 = (n+1)\left(\frac 1p -\frac 12\right) -1 \leq 0.$$

Using estimate \eqref{StrichMainEstim} in a classical argument to  establish well posedness results for  \eqref{MainEq} with nonlinearities of type \eqref{NonLin} in the energy space $H^1(\R^n) \cap L^{r+1}(\R^n) \times L^2(\R^n)$ would require estimating the $L^{pr}(\R^n)$ norms of the supposed  solution in terms of the $H^1(\R^n)$ norms. In other words, this would require the validity of the Sobolev embedding $H^1(\R^n) \subset L^{pr}(\R^n)$ or equivalently $pr \leq \frac{2n }{n-2}.$ Thus,   $r$ would have the upper bound 
$$r\leq  \frac{2n }{(n-2)p} \leq \frac{n (n+3) }{(n+1)(n-2)}= 1+ \frac{4n +2}{(n+1)(n-2)}.$$ In particular,  for $\dta >0,$ this upper bound can be a bit larger. Indeed,  ${\rm rank}({\rm Hess}(\lambda _\delta)) = n$ and thus the term $-\frac{n-1}{2}$ in Corollary \ref{CorLittman} can be replaced by  $-\frac n2$ which in terns leads to the lower bound $\frac{2n}{n+2}$ of $p$ and then we get $$r\leq 1+\frac{4}{n-2}.$$
This coincides with the nonlinearity considered by \cite{Inui:2019:StrichartzEstimWDWE} and \cite{Watanabe:2017:StrichartzEstimWDWEd23} for the weakly damped wave equations. 

Finally,  the main result can be extended to more general strongly damping cases for which the damping operator is given by
$$ A v = \mathfrak{F}^{-1}(a \hat v),$$
where $\xi \mapsto a(\xi)$ possesses the same properties of $\xi\mapsto \dta |\xi|^2$ (The Fourier multiplier of $-\dta \Delta$) that make the whole analysis  work.

\bibliographystyle{siam}
\bibliography{Dpdkg}

\end{document}